\documentclass{amsart}
\usepackage{amsfonts,amssymb,amscd,amsmath,enumerate,verbatim,calc}
\usepackage[all]{xy}
\newcommand{\CM}{Cohen-Macaulay}

\newcommand{\wrt}{with respect to}

\newcommand{\xb}{\mathbf{x} }
\newcommand{\yb}{\mathbf{y} }

\newcommand{\F}{\mathbb{F} }
\newcommand{\G}{\mathbb{G} }

\newcommand{\db}{\mathbf{d} }

\newcommand{\q}{\mathfrak{q} }
\newcommand{\A}{\mathfrak{a} }

\newcommand{\Z}{\mathbb{Z} }
\newcommand{\Pc}{\mathcal{P} }
\newcommand{\Sq}{\mathcal{S}{q}}
\newcommand{\Qc}{\mathcal{Q} }
\newcommand{\Hc}{\mathcal{H} }
\newcommand{\rt}{\rightarrow}

\newcommand{\ov}{\overline}

\newcommand{\wh}{\widehat }

\newcommand{\image}{\operatorname{image}}
\newcommand{\coker}{\operatorname{coker}}

\newcommand{\grade}{\operatorname{grade}}
\newcommand{\depth}{\operatorname{depth}}
\newcommand{\codim}{\operatorname{codim}}
\newcommand{\ann}{\operatorname{ann}}
\newcommand{\height}{\operatorname{height}}

\newcommand{\Spec}{\operatorname{Spec}}

\newcommand{\Sym}{\operatorname{Sym}}

\newcommand{\Ext}{\operatorname{Ext}}

\theoremstyle{plain}

\newtheorem{theorem}{Theorem}[section]
\newtheorem{corollary}[theorem]{Corollary}
\newtheorem{lemma}[theorem]{Lemma}
\newtheorem{proposition}[theorem]{Proposition}

\theoremstyle{definition}
\newtheorem{definition}[theorem]{Definition}
\newtheorem{s}[theorem]{}
\newtheorem{remark}[theorem]{Remark}

\theoremstyle{remark}

\begin{document}

\title[Annihilators of local cohomology modules]{Annihilators of local cohomology modules over modular invariant rings and Dickson polynomials}
\author{Tony~J.~Puthenpurakal}
\date{\today}
\address{Department of Mathematics, IIT Bombay, Powai, Mumbai 400 076, India}

\email{tputhen@gmail.com}

 \begin{abstract}
Let $\F_q$ be a finite field with $q = p^s$ elements. Let $V$ be a $d$ dimensional vector space over $\F_q$ and let $G$ be a subgroup of $GL(V)$. Let $R = \F_q[V] = \Sym_{\F_q}(V^*)$ and let $G$ act naturally on $R$. Set $S = R^G$. Let $\db_{d,0}, \db_{d, 1}, \ldots, \db_{d, d-1} \in S$ be the Dickson polynomials with $\deg \db_{d,i} = q^d - q^i$. Let $I$ be a homogeneous ideal of $S$ and let $H^i_I(S)$ be the $i^{th}$-local cohomology module of $S$ \wrt \ $I$. Let $J_i = \sqrt{\ann H^i_I(S)}$. Assume $J_i \neq 0$ and $\dim S/J_i = d - g$. Then we show that
$\db_{d,0}, \ldots, \db_{d, g - 1} \in J_i$. We give several applications of our results. An application is a considerably simpler proof of Landweber-Stong conjecture. We also give a necessary and sufficient condition for $S$ to have $S_r$-property of Serre in terms of Dickson  polynomials.
\end{abstract}
 \maketitle
\section{introduction}
\s \label{setup}Let $\F_q$ be a finite field with $q = p^s$ elements. Let $V$ be a $d$ dimensional vector space over $\F_q$ and let $G$ be a subgroup of $GL(V)$. Let $R = \F_q[V] = \Sym_{\F_q}(V^*)$ and let $G$ act naturally on $R$. Set $S = R^G$. We note that $GL(V)$ is also a finite group and so $S$ contains $D^*(d) = R^{GL(V)}$. Dickson proved that $D^*(d)$ is a polynomial ring with
$D^*(d) = \F_q[ \db_{d,0}, \db_{d, 1}, \ldots, \db_{d, d-1}]$ with $\deg \db_{d,i} = q^d - q^i$. We follow the exposition given in \cite{SL}. It is well-known that the  Dickson algebra plays a key role in the invariant theory of finite groups in characteristic $p > 0$.
Local cohomology plays an important role in commutative algebra and algebraic geometry. Let $I$ be a homogeneous ideal in $S$. Let $H^i_I(S)$ be the $i^{th}$ local cohomology module of $S$ \wrt \ $I$.

In this paper we prove the following result.
\begin{theorem}\label{main}
(with setup as in \ref{setup}).  Let  $J_i = \sqrt{\ann H^i_I(S)}$. Assume $J_i \neq 0$ and $\dim S/J_i = d - g$. Then
$\db_{d,0}, \ldots, \db_{d, g - 1} \in J_i$.
\end{theorem}

\emph{Applications:}

I. One of the landmark results in modular invariant theory is the proof  by D. Bourguiba and S. Zarati, see \cite{bz}, of
the Landweber-Stong conjecture \cite{LS}. This result states that $\depth S \geq r$ if and only if $\db_{d,d-1}, \cdots, \db_{d, d-r}$ is an $S$-regular sequence. In our paper we give an  alternate proof of this conjecture. We prove.
\begin{theorem}\label{ls}(with hypotheses as in \ref{setup}).  The following conditions are equivalent:
  \begin{enumerate}[\rm (i)]
    \item $\depth S \geq r $.
    \item $\db_{d,d-1}, \cdots, \db_{d, d-r}$ is an $S$-regular sequence.
  \end{enumerate}
\end{theorem}
\begin{remark}
D. Bourguiba and S. Zarati prove a considerably more general result than
the Landweber-Stong conjecture. However our proof of the original   Landweber-Stong conjecture follows easily from Theorem \ref{main} and is more direct than the proof of  Bourguiba and  Zarati.
\end{remark}
II. \emph{$S_r$ property of $S$}: \\
Recall a Noetherian ring $T$ satisfies $S_r$ property of Serre if
$$ \depth T_P \geq \min\{\height P, r \} \quad \text{for all primes $P$ of $T$}. $$
The ring $S$ is normal. So in particular it is $S_2$. There exists invariant rings $S$ that are not $S_3$. There are Steenrod operators $\Pc^i \colon S \rt S$ for $i \geq 0$, see \cite[Chapter 11]{SL}.
An ideal $I$ is said to $\Pc^*$ invariant if $\Pc^* I \subseteq I$. There exist only finitely many $\Pc^*$-invariant homogeneous  prime ideals in $S$, see \cite[11.4.8]{SL}. Surprisingly the $S_r$ property of $S$ is controlled by $\Pc^*$-invariant prime ideals in the following sense.
\begin{theorem}\label{sr}(with hypotheses as in \ref{setup}). Let $r \geq 3$. The following conditions are equivalent:
  \begin{enumerate}[\rm (i)]
    \item $S$ satisfies $S_r$ property.
    \item $\depth S_P \geq \min\{\height P, r \}$  for all $\Pc^*$-invariant homogeneous primes $P$ of $S$.
  \end{enumerate}
\end{theorem}

Recall that  $\depth S \geq r$ if and only if $\db_{d,d-1}, \cdots, \db_{d, d-r}$ is an $S$-regular sequence. A natural question is what happens when
$\db_{d,0}, \ldots, \db_{d, r-1}$ is an $S$-regular sequence.
We prove
\begin{corollary}
\label{sr-cor} (with hypotheses as in \ref{setup}) Let $r \geq 3$. The following assertions are equivalent:
\begin{enumerate}[\rm (i)]
  \item $S$ satisfies the $S_r$ condition of Serre.
  \item $\db_{d,0}, \ldots, \db_{d, r-1}$ is an $S$-regular sequence.
\end{enumerate}
\end{corollary}
Note that if $\dim S \leq 3$ then $S$ is \CM. When $d \geq 4$
 we prove.
\begin{corollary}
\label{cm-cor} (with hypotheses as in \ref{setup}). Let $d \geq 4$  The following assertions are equivalent:
\begin{enumerate}[\rm (i)]
  \item $S$ is \CM.
  \item $\db_{d,0}, \ldots, \db_{d, d-2}$ is an $S$-regular sequence.
\end{enumerate}
\end{corollary}
\begin{proof}
The assertion (i) $\implies$ (ii) is trivial. We prove the converse. By \ref{sr-cor} we get that $S$ satisfies $S_{d-1}$ property of Serre.  Then the non-CM locus of $S$ is contained in $\{S_+\}$. The assertion follows from a result due to Kemper which states that the non-Cohen-Macaulay locus of a ring of invariants is either empty or has dimension at least one, \cite[3.1]{K}.
\end{proof}

III \emph{Cohomological annihilators:}\\
The ideals $J_i = \sqrt{\ann H^i_{S_+}(S)}$  carry a lot of homological information. We first prove
\begin{theorem}\label{coho-ann}(with hypotheses as in \ref{setup}). Let $d \geq 4$. Then $J_i = S$ for $0 \leq i \leq 2$. Furthermore for $3 \leq i \leq d - 1$ we have
$\dim S/J_i \leq i - 2$.
\end{theorem}
As an application of Theorem \ref{main} we obtain
\begin{corollary}
\label{ann-coho-cor}(with hypotheses as in \ref{coho-ann}). We have $(\db_{d,0}, \db_{d, 1}, \db_{d,2}) \subseteq J_i$ when $0 \leq i \leq d-1$.
\end{corollary}
In \cite{P}, it is  proved that $\db_{d,0} \in J_i$ when $0 \leq i \leq d-1$. We obtain a curious consequence of Corollary \ref{ann-coho-cor}.
\begin{corollary}\label{khy}
(with hypotheses as in \ref{coho-ann}). Let $P$ be a   prime ideal in $S$ such that $S_P$ is not \CM. Then $(\db_{d,0}, \db_{d, 1}, \db_{d,2}) \subseteq P$.
\end{corollary}

Annihilators of local cohomology modules give a lot of information on annihilators of cohomology in many constructs. Set $\A = (\db_{d_0}, \db_{d,1}, \db_{d,2})$. For $0 \leq j \leq d - 1$, let $a_j \geq 0$ be the smallest power of $\A$ such that $\A^{a_j}\subseteq \ann_S H^j_{S_+}(S)$. Set $\q_i = \prod_{j = 0}^{i} \A^{a_j}$ for $0 \leq j \leq d -1$. Our first application is a precise version of a result of a result of Roberts, see \cite[Theorem 1]{R} (also see \cite[8.1.2]{BH}).
\begin{corollary}
\label{rob} (with hypotheses as in \ref{setup}). Let $\G = 0 \rt G^0 \rt G^1 \rt \ldots \rt G^r \rt 0$ be a complex of free graded $S$-modules and homogeneous maps such that $H^i(\G)$ has finite length for all $i$. Then $\q_i \subseteq \ann_S H^i(\G)$ for $0 \leq i \leq d-1$.
\end{corollary}
Our next application is a precise version of a result due to Schenzel, see \cite[Theorem 2]{Sc}  (also see \cite[8.1.4]{BH}).
\begin{corollary}
\label{Sch}(with hypotheses as in \ref{setup}). Let $\xb = x_1,\dots, x_d$ be a homogeneous system of parameters of $S$. Then
$$ \q_{d-1} \subseteq \ann_S \frac{(x_1, \ldots x_{t-1})\colon x_t}{(x_1, \ldots, x_{t-1})} \quad \text{for t = 1, \ldots, d}.$$
\end{corollary}
Finally we give a precise version of a result in \cite[8.1.3]{BH}.
\begin{corollary}
\label{Sch-bh}(with hypotheses as in \ref{setup}).
Let $\xb = x_1,\dots, x_m$ be a  set of homogeneous  elements in $S$ such that $\codim (\xb) = m$. Then for $i = 1, \ldots, m$ we get that $\q_{d-i}$ is contained in the annihilator of the Koszul homology $H_i(\xb)$.
\end{corollary}
 \s \emph{Technique to prove our results:}
 A homogeneous  ideal $I$ of ($R$ or $S$) is said to be $\Pc^*$-invariant if $\Pc^i(I) \subseteq I$ for all $i \geq 0$.  We note that if $I$ is $\Pc^*$-invariant ideal in $S$ then so is $\sqrt{I}$, see \cite[11.2.5]{SL}. It is easily proved that if $I$ is a non-zero $\Pc^*$-invariant ideal of $S$ with $\dim S/\sqrt{I} = d - g$ then $\db_{d,0}, \ldots, d_{d, g -1} \in \sqrt{I}$, see \ref{bella}.
\begin{definition}
By a Cartan $S$-module we mean a graded  $S$-module $M$ equipped with $\F_q$-linear maps $\Qc^i_M \colon M \rt M$ for $i \geq 0$ with $\Qc^0_M $ being the identity such that for any $s \in S$ and $m \in M$ we have for all $r \geq 1$
\begin{enumerate}
  \item $$ \Qc^r_M(sm) = \sum_{i+j = r}\Pc^i(s)\Qc^j_M(m).$$
  \item  $\Qc^r_M(M_n) \subseteq M_{n + r(q-1)} $ for all $n \in \Z$.
\end{enumerate}
\end{definition}

 It is easily seen that if $M$ is a Cartan $S$-module then $\ann_S M$ is a $\Pc^*$-invariant ideal of $S$, see \ref{ann-p*}. Theorem \ref{main} is a consequence of the following result:
\begin{theorem}\label{chief-intro}
(with setup as in \ref{setup}) Let $I$ be a homogeneous ideal in $S$. Let $I = (x_1, \ldots, x_m)$ with $x_j$ homogeneous. Fix $i \geq 0$. Then there exists a Cartan $S$-module structure on $H^i_I(S)$ (depending on $\xb$).
\end{theorem}

\begin{remark}
Previously unstable $\Pc^*$-modules were studied, for instance see \cite[8.5]{NS}. Recall a graded Cartan module $M$ is said to be unstable if $\Pc^i(m) = 0$ for all homogeneous $m \in M$ with $i > \deg m$.  We note that usually most of $H^i_{S_+}(S)$ is concentrated in negative degrees. So unstable structures might not exist  for local cohomology modules.
\end{remark}

Our proof of Landweber-Stong conjecture follows from Theorem \ref{main} and  the structure of $\Pc^*$-invariant prime  ideals of $D^*(d)$. For application II we use a curious property of homogeneous prime ideals $P$ which are \textit{not} $\Pc^*$-invariant: if $H^i_P(S)$ is a finitely generated $S$-module then $H^i_P(S)_P = 0$. Theorem \ref{coho-ann} is an application of the Ellingsrud-Skjelbred spectral sequences \cite{ES}. The applications \ref{rob}, \ref{Sch} and \ref{Sch-bh} follows from standard results as given in \cite[Chapter 8]{BH}.

We now describe in brief the contents of this paper.  In section two we discuss a few preliminary results that we need. In Section three we discuss our notion of Cartan modules and discuss a few preliminary properties. In section four we discuss a Cartan $S$-module structure on $S_x$ where $x \in S$ is homogeneous.  In section five we discuss Cartan $S$-module structure on local cohomology modules.   In the next section we give a proof of Theorem \ref{main}. In section seven we give our proof of Landweber-Stong conjecture. In the next section we give proof of our applications II. Finally in section nine we give a proof of Theorem \ref{coho-ann} and Corollary \ref{khy}.
In the appendix we discuss a result that we need.
\section{Preliminaries}
In this section we discuss a few preliminary results that we need. \\
\textbf{I:} \emph{Ellingsrud-Skjelbred spectral sequences.}

In this sub-section we describe the Ellingsrud-Skjelbred spectral sequences \cite{ES} (we follow the exposition given in \cite[8.6]{Lorenz}).
\s Let $S$ be a commutative Noetherian ring. Let $Mod(S)$ be the category of left $S$-modules.

(1) Let $G$ be a finite group. Let $S[G]$ be the group ring and let $Mod(S[G])$ be the category of left $S[G]$-modules.

(2) Let $\A$ be an ideal in $S$. Let $\Gamma_\A(-)$ be the torsion functor associated to $\A$.  Let $H^n_\A(-)$ be the $n^{th}$ right derived functor of
$\Gamma_\A(-)$. Usually $H^n_\A(-)$ is called the $n^{th}$ local cohomology functor of $S$ \wrt \ $\A$.

(3) If $M \in Mod(S[G])$ then note $\Gamma_\A(M) \in Mod(S[G])$.

\s Ellingsrud-Skjelbred spectral sequences  are constructed as follows: Consider the following sequence of functors
\[
\rm{(i)} \quad   Mod(S[G]) \xrightarrow{(-)^G} Mod(S) \xrightarrow{\Gamma_\A} Mod(S),
\]
\[
\rm{(ii)}   \quad   Mod(S[G]) \xrightarrow{\Gamma_\A} Mod(S[G]) \xrightarrow{(-)^G} Mod(S)
\]
We then notice
\begin{enumerate}[\rm (a)]
\item
The above compositions are equal.
\item
It is possible to use Grothendieck spectral sequence of composite of functors to both (i) and (ii) above; see
\cite[8.6.2]{Lorenz}.
\end{enumerate}
Following Ellingsrud-Skjelbred  we let $H^n_\A(G,-)$ denote the $n^{th}$ right derived functor of this composite
functor. So by (i) and (ii) we have two first quadrant spectral sequences for each $S[G]$-module $M$
\[
(\alpha)\colon \quad \quad  E_2^{p,q} = H^p_\A(H^q(G, M)) \Longrightarrow H^{p+q}_\A(G,M), \ \text{and}
\]
\[
(\beta)\colon \quad \quad \mathcal{E}_2^{p,q} = H^p(G, H^q_\A(M)) \Longrightarrow H^{p+q}_\A(G,M).
\]
\begin{remark}\label{grading}
(1) If $S$ is $\mathbb{N}$-graded ring then $S[G]$ is also $\mathbb{N}$-graded (with $\deg \sigma = 0$ for all $\sigma \in G$).

(2)  If $M$ is a finitely generated graded left $S[G]$-module then $H^n(G, M)$ are finitely generated graded $S$-module. This can be easily seen by taking a graded free resolution of $S$ consisting of finitely generated free  graded $S[G]$-modules.

(3) Ellingsrud-Skjelbred spectral sequences have an obvious graded analogue.
\end{remark}

\textbf{II:} \emph{Steenrod operators}.

We follow the exposition given in \cite{SL}.
\s Let $p$ be a prime, $q = p^s$ and let $V$ be a $\F_q$-vector space of dimension $d$.
Define $P(\xi)\colon \F_q[V] \rt \F_q[V][\xi]$ (with $\deg \xi = 1-q$) by the rules:
\begin{enumerate}
  \item $P(\xi)(v)  = v + v^q\xi$  for all $v \in  V^*$.
  \item $P(\xi)(u + w) = P(\xi)(u) + P(\xi)(w)$ for all $u,w \in F_q[V]$.
  \item $P(\xi)(uw) = P(\xi)(u)P(\xi)(w)$ for all $u,w \in F_q[V]$.
  \item $P(\xi)(1) = 1$.
   \end{enumerate}
  We note that $P(\xi)$ is a ring homomorphism of degree zero.

By separating out homogeneous components we obtain $\F_q$ linear maps \\ $\Pc^i \colon \F_q[V] \rt \F_q[V]$ by the requirement
$$P(\xi)(f) = \sum_{i \geq 0}\Pc^i(f)\xi^i.$$
Note $\Pc^0$ is the identity.
The operations $\Pc^i$ are called the \emph{Steenrod reduced power operations} over $\F_q$ and when $p = 2$, also denoted by $\Sq^i$ and referred to as \emph{Steenrod squaring operations}.
We will use the following Cartan formulae:
$$\Pc^k(uv) = \sum_{i+j = k}\Pc^i(u)\Pc^j(v). $$

We  have for homogeneous $u$, $\Pc^{\deg u}(u) = u^q$ and $\Pc^{i}(u) =  0$ for $i > \deg u$.
We also have that $\Pc^r(\F_q[V]_n) \subseteq \F_q[V]_{n + r(q-1)}$ for all $n \in \Z$.

\s Let $G \subseteq GL(V)$ be a finite group. Let $S = \F_q[V]^G$. We define an action of $G$ on $\F_q[V][\xi]$ by giving the trivial action on $\xi$. We note that $\sigma(P(\xi)(r)) = P(\xi)(\sigma(r))$ for all $\sigma \in G$. Thus the maps $\Pc^i$ restrict to maps from $S$ to $S$.

\s An ideal $I$ of $S$ is said to be \emph{$\Pc^*$-invariant} if $\Pc^i(I)\subseteq I$ for all $i \geq 0$.

\s \label{dickson} The non-zero $\Pc^*$-prime ideals of $D^*(d)$ are $(\db_{d,0}, \cdots, \db_{d,i})$ for $i = 0, \ldots, d -1$; see \cite[11.4.6]{SL}.
We need the following result.
\begin{lemma}\label{bella}
(with setup as in \ref{setup}). Let $J$ be a non-zero  radical $\Pc^*$-invariant ideal of $S$. If $\dim S/J = d - g$ then $\db_{d,0}, \cdots, \db_{d, g -1} \in J$.
\end{lemma}
\begin{proof}
  We have $D^*(d) \subseteq S$ is a finite extension of normal domains. Then $K =  J\cap D^*(d)$ is a $\Pc^*$-invariant ideal of $D^*(d)$. We note that $S/J$ is a finite integral extension of $D^*(d)/K$. So $\dim D^*(d)/K = d - g$. As $K$ is a radical $\Pc^*$-invariant ideal of $D^*(d)$ it follows that $K$ is a prime ideal, \cite[11.4.7]{SL}. Also $\height K = g$ as $D^*(d)$ is \CM. Therefore $K = (\db_{d,0}, \cdots, \db_{d, g -1})$ by \ref{dickson}. The result follows.
\end{proof}

\section{Cartan Modules}
In this section we describe our notion of Cartan modules. We work with the setup as in \ref{setup}. Let $\Pc^i \colon S \rt S$ be the $i^{th}$-Steenrod operator on $S$.

\begin{definition}
By a Cartan $S$-module we mean a graded $S$-module $M$ equipped with $\F_q$-linear maps $\Qc^i_M \colon M \rt M$ for $i \geq 0$ with $\Qc^0_M $ being the identity such that for any $s \in S$ and $m \in M$ we have for all $r \geq 1$
\begin{enumerate}
  \item $$ \Qc^r_M(sm) = \sum_{i+j = r}\Pc^i(s)\Qc^j_M(m).$$
  \item $\Qc^r_M(M_n) \subseteq M_{n + r(q-1)}$ for all $n \in \Z$.
\end{enumerate}
\end{definition}
Next we define morphisms of Cartan $S$-modules.
\begin{definition}\label{def-cart-morph}
  Let $M, N$ be Cartan $S$-modules. By a morphism $f \colon M \rt N$ of Cartan $S$-modules we mean a homogeneous $S$-linear map $f \colon M \rt N$ such that the following diagram commutes for all $i \geq 0$.

\[
  \xymatrix
{
M
\ar@{->}[r]^{f}
\ar@{->}[d]^{\Qc_M^i}
 & N
\ar@{->}[d]^{\Qc_N^i}
\\
M
\ar@{->}[r]^{f}
 & N
\
 }
\]
\end{definition}
The following result follows easily from the definition of morphism of Cartan $S$-modules.
\begin{proposition}
\label{kic} Let $M, N$ be Cartan $S$-modules and let $f \colon M \rt N$ be a morphism of Cartan $S$-modules. Then $\ker f, \image f, \coker f$ are Cartan $S$-modules. Furthermore the natural inclusions $\ker f \rt M$, $\image f \rt N$ and the surjection $N \rt \coker f$ are morphisms of Cartan $S$-modules. \qed
\end{proposition}
As a consequence we obtain
\begin{corollary}
\label{complex-cart} Let $\G \colon M_1 \xrightarrow{f_1} M_2 \xrightarrow{f_2} M_3$ be a complex of Cartan $S$-modules and morphisms $f_1, f_2$ of Cartan $S$-modules. Then $H^2(\G)$ is a Cartan $S$-module.
\end{corollary}
\begin{proof}
  By \ref{kic} we have $\ker f_2$ and $\image f_1$ are Cartan submodules of $M_2$. So again by \ref{kic}, $H^2(\G) = \ker f_2/\image f_1$ is a Cartan $S$-module.
\end{proof}

The following Lemma is critical.
\begin{lemma}\label{ann-p*}
Let $M$ be a Cartan $S$-module. Then $\ann M$ is a $\Pc^*$-invariant ideal of $S$.
\end{lemma}
\begin{proof}
Let $s \in \ann M$. We show $\Pc^i(s) \in \ann M$ for all $i \geq 0$. This is certainly true when $i = 0$ as $\Pc^0$ is the identity. Let $m \in M$.
Then we have
$$ 0 = \Qc^1(s m) = \Pc^1(s) m + s \Qc^1(m) = \Pc^1(s)m.$$
The third equality is as $\Qc^1(m) \in M$. Thus $\Pc^1(s) \in \ann M$.

Assume $\Pc^i(s) \in \ann M $ for $i = 0, \ldots, r-1$ where $r \geq 2$. Then we have
$$0 = \Qc^r(sm) = \sum_{i+j = r}\Pc^i(s)\Qc^j(m) = \Pc^r(s)m. $$
The third equality is because $\Pc^i(s) \in \ann M$ for $i< r$. By the above equality we get $\Pc^r(s) \in \ann M$.
The result follows.
\end{proof}
\section{The Cartan $S$-module structure on $S_x$}
In this section we prove two critical results. We first prove
\begin{theorem}
\label{loc}(with setup as in \ref{setup}). Let $x \in S$ be homogeneous and non-zero. Then there exists a  unique Cartan S-module structure on $M = S_x$ such that $\Qc^i_M(s) = \Pc^i(s)$ for all $s \in S$ and for all $i \geq 0$.
\end{theorem}
The next result is also critical.
\begin{theorem}
\label{maps}(with setup as in \ref{setup}). Let $x, y$ be homogeneous and non-zero. Then the natural map $S_x \rt S_{xy}$ is morphism of Cartan $S$-modules (the Cartan module structures on $S_x, S_{xy}$ are as constructed in Theorem \ref{loc}).
\end{theorem}

We first give
\begin{proof}[Proof of Theorem \ref{loc}]
Let $\Pc(\xi)\colon S \rt S[\xi]$ denote  the total Steenrod operation. Since
$$\Pc(\xi)(x) = x + \ \text{terms of positive $\xi$-degree,}$$
its constant term is $x$ which is invertible in $S_x$. Therefore $\Pc(\xi)(x)$ is a unit in $S_x[[\xi]]$. Therefore $\Pc(\xi)$  extends uniquely to a ring homomorphism $\Qc(\xi) \colon S_x \rt S_x[[ \xi]]$ defined by
$$ \Qc(\xi)(\frac{a}{x^m}) = \Pc(\xi)(a)\Pc(\xi)(x)^{-m}.$$
Set $u = a/x^m$ and
$$ \Qc(\xi)(u) = \sum_{n \geq 0}\Qc^n(u)\xi^n.$$
As $\Qc(\xi)$ is a ring homomorphism it follows easily that $\Qc^n \colon S_x \rt S_x$ are $\F_q$-linear maps.
As $\Pc(\xi)(a) = \Pc(\xi)(x^m)\Qc(\xi)(u),$ we obtain by comparing coefficients that
$$ \Pc^n(a) = \sum_{i+j = n}\Pc^i(x^m)\Qc^j(u) \quad \text{for all $n\geq 0$}.$$
We get $\Qc^0(u) = u$. We also note that for $a \in S$ we have $\Qc(\xi)(a) = \Pc(\xi)(a)$. So $\Qc^i(a) = \Pc^i(a)$ for all $i \geq 0$.

Next we verify the Cartan identity. Let $t\in S$  and let $u \in S_x$. We have
\begin{align*}
  \sum_{n \geq 0}\Qc^n(tu)\xi^n &= \Qc(\xi)(tu), \\
  &= \Qc(\xi)(t)\Qc(\xi)(u) \\
   &= \Pc(\xi)(t) \Qc(\xi)(u), \\
  &= \left(\sum_{n \geq 0}\Pc^n(t)\xi^n \right)\left( \sum_{n \geq 0}\Qc^n(u)\xi^n \right), \\
  &= \sum_{n \geq 0}\left(\sum_{i+j = n}\Pc^i(t)\Qc^j(u) \right) \xi^n.
\end{align*}
The Cartan identity follows.
Next we show that for all $r \geq 1$ we have $\Qc^r((S_x)_n) \subseteq (S_x)_{n + r(q-1)}$.
Let $u = a/x^m$ where $a$ is homogeneous and $|u| = |a| - m|x| = n$. We prove by induction on $r$ that $\Qc^r(u) \in (S_x)_{n + r(q-1)}$. We have nothing to prove when $r = 0$ as $\Qc^0$ is the identity. Assume $r \geq 1$ and that we have $\Qc^i(u) \in (S_x)_{n + i(q-1)}$ for $i < r$.
We have

$$ \Pc^r(a) = \sum_{i+j = r}\Pc^i(x^m)\Qc^j(u)  = \sum_{i+j = r, j<r}\Pc^i(x^m)\Qc^j(u) + x^m\Qc^r(u). $$
We have $\Pc^r(a) \in S_{|a| + r(q-1)}$. We also have  for $j < r$ by induction hypotheses that $$\Pc^i(x^m)\Qc^j(u) \in S_{m|x| + i(q-1)}((S_x)_{|u| +j (q-1)}) \subseteq (S_x)_{|a| + r(q-1)}. $$
So
$x^m\Qc^r(u) \in
(S_x)_{|a| + r(q-1)}$. Thus $\Qc^r(u) \in (S_x)_{|u| + r(q-1)}$. The result follows.

 Thus $\Qc^i \colon S_x \rt S_x$ for $i \geq 0$ defines a Cartan module structure on $S_x$.

Finally we show uniqueness. Let $\Hc^n \colon S_x \rt S_x$  for $n \geq 0, $ define another Cartan module structure on $S_x$ with $\Hc^n(a) = \Pc^n(a)$ for all $a \in S$. We prove by induction on $n$ that $\Hc^n = \Qc^n$. For $n = 0$ we have nothing to show. Assume the result for $n \leq r -1$. We prove for $n = r$. Let $u = a/x^m$. Then $a = x^mu$. So
\begin{align*}
  \Pc^r(a) = \Hc^r(a)  &= \sum_{i+j = r}\Pc^i(x^m)\Hc^j(u), \\
  &= x^m\Hc^r(u) +  \sum_{i+j = r, j<r}\Pc^i(x^m)\Hc^j(u),\\
  &= x^m\Hc^r(u) +  \sum_{i+j = r, j<r}\Pc^i(x^m)\Qc^j(u); \ \text{by induction hypothesis}, \\
   & = x^m\Hc^r(u)  + \Pc^r(a) - x^m\Qc^r(u).
\end{align*}
Thus $\Hc^r(u) = \Qc^r(u)$. The result follows by induction.
\end{proof}

Next we give
\begin{proof}
[Proof of Theorem \ref{maps}] Set $M = S_x$ and $N = S_{xy}$. Let $\eta \colon M \rt N$ be the natural map. So $\eta$ induces a natural map $\delta \colon M[[\xi]] \rt N[[\xi]]$. Let $\Qc_M(\xi) \colon M \rt M[[\xi]]$ and $\Qc_N(\xi) \colon \Qc_N \rt \Qc_N[[\xi]]$ be the maps constructed in proof of Theorem \ref{loc}. We have
$$\Qc_N(\xi)(\frac{a}{x^my^m}) = \Pc(\xi)(a)\Pc(\xi)(xy)^{-m}.$$
Now consider $u = a/x^m \in M$. Then $\eta(u) = ay^m/x^my^m \in N$. Then noting that $\Pc(\xi)$ is a ring homomorphism we get
\begin{align*}
  \Qc_N(\xi)(u) &=  \Pc(\xi)(ay^m)\Pc(\xi)(xy)^{-m},\\
  &=  \Pc(\xi)(a)\Pc(\xi)(y^m) \Pc(\xi)(x)^{-m}\Pc(\xi)(y)^{-m}\\
  &= \Pc(\xi)(a)\Pc(\xi)(x)^{-m} \\
  &= \delta(\Qc_M(\xi)(u)).
\end{align*}
The result follows.
\end{proof}

\section{Cartan module structure on local cohomology}
In this section we prove the following result:
\begin{theorem}\label{chief}
(with setup as in \ref{setup}) Let $I$ be a homogeneous ideal in $S$. Let $I = (x_1, \ldots, x_m)$ with $x_i$ homogeneous. Fix $i \geq 0$. Then there exists a Cartan $S$-module structure on $H^i_I(S)$ (depending on $\xb$).
\end{theorem}
We will need a few preliminaries.
\begin{proposition}\label{add}
Let $M, N$ be Cartan $S$-modules. Then $M \oplus N$  also has a natural Cartan $S$-module structure such that
\begin{enumerate}[\rm (1)]
  \item the natural maps $M \rt M\oplus N$, $N \rt M\oplus N$, $M\oplus N \rt M$ and $M\oplus N \rt N$ are morphisms of Cartan $S$-modules.
  \item if $U$ is a Cartan $S$-module and $f\colon M \rt U$ and $g \colon N\rt U$ are morphism of Cartan $S$-modules then $f\oplus g \colon M\oplus N \rt U$ is a morphism of Cartan $S$-modules.
  \item if $U$ is a Cartan $S$-module  and $\alpha \colon U \rt M$ and $\beta \colon U \rt N$ are morphism of Cartan $S$-modules then $(\alpha, \beta) \colon  U \rt M\oplus N$ is a morphism Cartan $S$-modules.
\end{enumerate}
\end{proposition}
\begin{proof}
Define $\Qc^r_{M\oplus N} \colon M\oplus N \rt M\oplus N$ to be
$$\Qc^r_{M\oplus N}(m, n) = (\Qc^r(m), \Qc^r(n)).$$
It is readily verified that this gives a Cartan $S$-module structure on $M\oplus N$.

The assertions (1), (2) and (3) hold trivially.
\end{proof}
We also need the following easily proved result.
\begin{proposition}\label{map-add}
Let $M_i, N_j$ be Cartan $S$-modules for $1 \leq i \leq m$ and $1\leq j \leq n$. Set $M = \bigoplus_{i = 1}^{m}M_i$ and $N = \bigoplus_{j = 1}^{n}N_j$. Give $M, N$ the induced Cartan $S$-module structure. Let $\psi = (\psi_{i,j})\colon M \rt N$ where $\psi_{i,j} \colon M_i \rt N_j$. If each $\psi_{i,j}$ is morphism of Cartan $S$-modules then $\psi$ is a morphism of Cartan $S$-modules. \qed
\end{proposition}
We now give
\begin{proof}[Proof of Theorem \ref{chief}]
Let $\mathcal{C} = 0 \rt C^0 \rt C^1 \rt \cdots \rt C^m \rt 0$ be the modified \v{C}ech complex on $\xb$, see \cite[5.1.5]{bs}. Here
$$ C^t = \bigoplus_{1 \leq i_1< i_2 < \cdots<i_t \leq m}S_{x_{i_1}x_{i_2}\cdots x_{i_t}},$$
where the differential $\partial^t \colon C^t \rt C^{t+1}$ is given by \\  $(-1)^{s-1}.\text{nat}\colon S_{x_{i_1}x_{i_2}\cdots x_{i_t}} \rt (S_{x_{i_1}x_{i_2}\cdots x_{i_t}})_{x_{j_s}}$ if
$\{ i_1, \ldots, i_t \} = \{ j_1, \ldots, \wh{j_s}, \ldots, j_{t+1} \}$ and $0$ otherwise. It follows from \ref{loc} and \ref{add} that $C^t$ is a Cartan $S$-module for all $t$. By \ref{maps} and \ref{map-add} it follows that $\partial^t$ is a morphism of Cartan modules. By \ref{complex-cart} it follows that $H^i_I(S)$ has a structure of a Cartan $S$-module.
\end{proof}

\begin{remark}
Suppose $\xb$ and $\yb$ be two different sets of homogeneous generators of $I$. Then  we can give a Cartan $S$-module structure on $H^i_I(S)$ which depends on $\xb$ and another Cartan $S$-module structure which depends on $\yb$. We \emph{do not know} whether these structures are the same. For our applications it suffices to know there is at least one Cartan $S$-module structure on local cohomology modules.
\end{remark}

\section{Proof of Theorem \ref{main}}
In this section we give:
\begin{proof}[Proof of Theorem \ref{main}]
By \ref{chief} we get that $H^i_I(S)$ is a Cartan $S$-module. So by \ref{ann-p*} we get that $\ann H^i_I(S)$ is a $\Pc^*$-invariant ideal of $S$. So $J_i = \sqrt{\ann H^i_I(S)}$ is also $\Pc^*$-invariant ideal of $S$, see \cite[11.2.5]{SL}.  The result now follows from \ref{bella}.
\end{proof}
\section{Proof of Landweber-Stong conjecture }
In this section we give a proof of Landweber-Stong conjecture.
\begin{proof}[Proof of Theorem \ref{ls}]
The assertion (ii) $\implies $ (i) is trivial. We prove the converse.
We note that $\db_{d,0},\ldots, \db_{d, d-1}$ is a system of parameters of $S$. So if $S$ is \CM \ then
we have nothing to prove.

Assume $\depth S = d-g$ where $g > 0$. Also as $S$ is normal we have $d - g \geq 2$. Let $D = D^*(d)$ be the Dickson algebra. We consider $S$ as a $D$-module. Set $J =  (\db_{d,d-1}, \cdots, \db_{d,g})$. We note that it suffices to prove $\depth S_P \geq d -g$ for all primes $P$ in $D$ which contain $J$, see \cite[1.2.10]{BH}. The assertion holds if $P = D_+$ by our assumption. Now let $P \supseteq J$ and $P \neq D_+$. Let $D(u)$ be the $*$-canonical module of $D$.
We note that $\dim H^{i}_{S_+}(S)^\vee = \Ext^{d-i}_D(S, D(u))$ has dimension $\leq i$
 (see \ref{dual-dim}).  Set $K_{d-i} = \ann_D \Ext^{d-i}_D(S, D(u))$. So
 $\db_{d,0}, \ldots, \db_{d, d-i-1} \in \sqrt{K_{d-i}}$ by Theorem \ref{main} for $i < d$.

  Let $P \supseteq J$ be minimal. Then $P = J$. As $\db_{d, 0} \notin P$ we get that
  $\Ext^{i}_D(S, D(u))_P = 0$ for $i > 0$. So $S_P$ is \CM \ and so has $\depth  = d -g$.

  Let $P \supseteq J$ be such that $\height P = d - g + t$ with $t \geq 1$ and that $P \neq D_+$. We assert that $\depth S_P \geq d- g$. If not then $\Ext^{d-g + t -(d -g -i)}_D(S, D(u))_P \neq 0$ for some $i > 0$. So we have $\Ext^{t+i}(S, D(u))_P \neq 0$. We have $U = \Ext^{t+i}(S, D(u)) = H^{d - t - i}_{S_+}(S)^\vee$. So $\db_{d,0}, \cdots, \db_{d, t+i -1} \in \sqrt{\ann_D U}$ by Theorem \ref{main}. As $U_P \neq 0$ we get that $\db_{d,0}, \cdots, \db_{d, t+i -1} \in P$. If $t+i -1 \geq g$ then note that $D_+ \subseteq P$ which is a contradiction. So  $t + i -1 < g$.
  Thus  $P$ contains $d - g + t + i$ variables where $i \geq 1$. This is a contradiction as $\height P = d - g + t$. The result follows.
\end{proof}
\section{Proof of Applications II}
In this section we give all the proofs of results stated in our Application II.
We first prove
\begin{proposition}
\label{rachel}(with hypotheses as in \ref{setup}). Let $P$ be a prime ideal in $S$ with $\height P \geq 2$. Then $H^0_P(S) = H^1_P(S) = 0$.  If $\height P \geq 3$ then $H^2_P(S)$ is a finitely generated $S$-module.
\end{proposition}
\begin{proof}
We use the Ellingsrud-Skjelbred  spectral sequence with $\A = P$ and $M = R$. We have two first quadrant spectral sequences
\[
(\alpha)\colon \quad \quad  E_2^{p,q} = H^p_{P}(H^q(G, R)) \Longrightarrow H^{p+q}_{P}(G,R), \ \text{and}
\]
\[
(\beta)\colon \quad \quad \mathcal{E}_2^{p,q} = H^p(G, H^q_{P}(R)) \Longrightarrow H^{p+q}_{P}(G,R).
\]
We first analyse the spectral sequence $(\beta)$. We note that as $R$ is a finite $S$-module we have $H^q_{P}(R) = H^q_{PR}(R)$. We note that $PR$ has height $\geq 2$. So we have $H^i_{PR}(R) = 0$
for $i = 0, 1$.  If $\height P \geq 3$ then $\height PR \geq 3$. So $H^2_{PR}(R) = 0$. In particular $H^n_{P}(G, R) = 0$ for $n < 2$ and $H^2_{P}(G, R) = 0$ if $\height P \geq 3$.

Next we analyse the spectral sequence $(\alpha)$. We note that $E_2^{0,0} = H^0_P(S) = E_\infty^{0,0} = 0$. We also have $E_2^{1,0} = H^1_P(S) = E_\infty^{1,0} = 0$. Now assume $\height P \geq 3$.
We have an exact sequence
$$ E_2^{0, 1} = H^0_P(H^1(G, R))  \rt E_2^{2, 0} = H^2_P(S) \rt E_3^{2,0}.$$
We note that $E_3^{2,0} = E_\infty^{2,0} = 0$ as it is sub-quotient of $H^2_{P}(G, R) = 0$. So $H^2_P(S)$ is a quotient of $H^0_P(H^1(G, R))$ which is a finitely generated $S$-module.
\end{proof}
We will need the following:
\begin{lemma}
\label{eclair} Let $M$ be a Cartan $S$-module. Assume $P$ is not a $\Pc^*$-invariant prime ideal. If $P^r \subseteq \ann M$ then $M_P = 0$.
\end{lemma}
\begin{proof}
By \ref{ann-p*} we get that $J = \ann M$ is a $\Pc^*$-invariant ideal of $S$. We also have $\sqrt{J}$ is a $\Pc^*$-invariant ideal in $S$, see \cite[11.2.5]{SL}. As $\sqrt{J}  \supseteq P$ and $P$ is not $\Pc^*$-invariant ideal in $S$ it follows that there exist $\xi \in \sqrt{J} \setminus P$. As $\xi$ is invertible in $S_P$ and as some power of it annihilates $M_P$, it follows that $M_P = 0$.
\end{proof}
The following result follows easily from Theorem \ref{chief} and \ref{eclair}.
\begin{corollary}
\label{stock}(with setup as in \ref{setup}). Let $P$ be a prime ideal in $S$ which is \emph{not} $\Pc^*$-invariant. If $H^i_P(S)$ is a finitely generated module as an $S$-module then $H^i_P(S)_P = 0$. \qed
\end{corollary}
The following result is crucial for our applications.
\begin{theorem}\label{crucial}
(with setup as in \ref{setup}). Let $r \geq 3$. The following assertions are equivalent:
\begin{enumerate}[\rm (i)]
\item
$\depth S_P \geq \min \{ \height P, r \}$ for all homogeneous primes $P$ of $S$.
\item
$\depth S_P \geq \min \{ \height P, r \}$ for all homogeneous $\Pc^*$-invariant primes $P$ of $S$.
\end{enumerate}
\end{theorem}
\begin{proof}
The assertion (i) $\implies$ (ii) is trivial.
We now prove (ii) $\implies$ (i). Let $P$ be a homogeneous  prime ideal of $S$. If $\height P \leq 2$ then clearly  $\depth S_P = \height P$, by \ref{rachel}.
Assume that $\height P  = t \geq 3$. Set
$$f_P = \min \{ i \mid H^i_P(S) \ \text{is not finitely generated $S$-module} \}. $$
Set $s_P = \min \{ \height P, r \}$.
We make the following

Claim: If $P$ is  a homogeneous prime ideal then $f_P \geq  s_P$.

We first prove that the Claim proves the result. Let  $P$ be not a $\Pc^*$-invariant homogeneous  ideal. Then $H^i_P(S)$ are finitely generated $S$-modules for $i < s_P$.  By \ref{stock} it follows that $H^i_{PS_P}(S_P) = H^i_{P}(S)_P = 0$ for $i < s_P$. Thus $\depth S_P \geq s_P$. The result follows.

We prove the Claim by induction on $r \geq 3$ (we are assuming that $\height P = t \geq 3$). \\
We first consider the case when $r = 3$. In this case the result follows from Proposition \ref{rachel}.

Next consider the case when $r \geq 4$ and the result holds for $r - 1$.
By induction hypothesis, $f_P(S) \geq \min \{ \height P, r - 1 \}$ for all homogeneous primes.   In particular $\depth S_\q \geq  \min \{\height \q, r - 1 \}$ for all homogeneous non $\Pc^*$-invariant prime ideals $\q$ of $S$. We also have $\depth  S_Q \geq s_Q$ for all $\Pc^*$-invariant prime ideals $Q$ by our assumption.
We may assume $\height P \geq r$, otherwise there is nothing to show.
 By \cite[13.1.17]{bs} we have
\[
f_P = \min \{ \depth S_\q  + \height((P+\q)/\q) \mid \q \in \  ^* \Spec(S) \setminus V(P) \}.
\]
If $\height \q \leq r -1$ then $\depth S_\q = \height \q$. We have $\depth S_\q + \height((P+\q)/\q) = \height \q + \height(P + \q) - \height \q = \height (P + \q) \geq r$.
If $\height \q \geq r$ then $$\depth S_\q  + \height((P+\q)/\q)  \geq r - 1 + 1 = r.$$
Thus $f_P \geq r$. So Claim follows.
\end{proof}

We now give
\begin{proof}[Proof of Theorem \ref{sr}]
The assertion (i) $\implies$ (ii) is clear.

We now prove (ii) $\implies$ (i). It suffices to prove $\depth S_P \geq \min \{\height P, r \}$ for homogeneous primes in $S$, see \ref{hom-sr}.
By \ref{crucial} it follows that \\ $\depth S_P \geq \min \{ \height P, r \}$ for all homogeneous primes $P$ of $S$. The result follows.
\end{proof}

Next we give
\begin{proof}[Proof of Corollary \ref{sr-cor}]
We first prove (ii) $\implies $ (i). Let $P$ be a graded $\Pc^*$-invariant prime ideal of $S$. We note that if  $\height P = g$ then $\dim S/P = d - g$, as $S$ is equi-dimensional and universally catenary, see \cite[Lemma 2, p.\ 250]{Ma} for the local version; same proof works for the *-local case. By \ref{bella} we get that if $g \leq r$ then $P$ contains $\db_{d,0}, \ldots, \db_{d,g-1}$.
It follows that $S_P$ is \CM. If $\height P \geq r + 1$ then by \ref{bella}  we get that  $P$ contains $\db_{d,0}, \ldots, \db_{d,r-1}$.
So $\depth S_P \geq r$. The result follows from \ref{sr}.

Next we prove  (i) $\implies $ (ii). The extension $D^*(d) \subseteq S$ is an extension of normal domains. As $(\db_{d,0}, \ldots, \db_{d,r-1})$ is a prime ideal of height $r$ in $D^*(d)$, it follows that $(\db_{d,0}, \ldots, \db_{d,r-1})S$ has height $r$. As $S$ satisfies $S_r$ it follows from \cite[1.2.10]{BH} that $\grade \left( (\db_{d,0}, \ldots, \db_{d,r-1})S, S \right)) = r$. The result follows.
\end{proof}

\section{Application III}
Set dimension of the zero module as $-1$.
We first give
\begin{proof}[Proof of Theorem \ref{coho-ann}]
We note that if $d \leq 3$ then $S$ is \CM \ (see \cite{S-3}) and so we have nothing to prove. Also if $p$ does not divide order of $G$ then $S$ is \CM, and so we have nothing to prove. So we assume that $p$ divides $|G|$.  If $d \geq 4$ then by a result of Ellingsrud-Skjelbred \cite[3.9]{ES} we have
\[
\depth S \geq \min\{ \dim_{\F_q}V^P + 2, d \},
\]
where $P$ is a Sylow $p$-subgroup of $G$. We note that $\dim_{\F_q} (V)^P \geq 1$, see \cite[8.2.1]{SL}. So $\depth S \geq 3$. Thus $H^j_{S_+}(S) = 0$ for $j \leq 2$.

We use the Ellingsrud-Skjelbred  spectral sequence with $\A = S_+$ and $M = R$. We have two first quadrant spectral sequences
\[
(\alpha)\colon \quad \quad  E_2^{p,q} = H^p_{S_+}(H^q(G, R)) \Longrightarrow H^{p+q}_{S_+}(G,R), \ \text{and}
\]
\[
(\beta)\colon \quad \quad \mathcal{E}_2^{p,q} = H^p(G, H^q_{S_+}(R)) \Longrightarrow H^{p+q}_{S_+}(G,R).
\]
We first analyse the spectral sequence $(\beta)$. We note that as $R$ is a finite $S$-module we have $H^q_{S_+}(R) = H^q_{R_+}(R) = 0$ for $q \neq d$. So the spectral sequence $(\beta)$ collapses. In particular $H^n_{S_+}(G, R) = 0$ for $n < d$.

Next we analyse the spectral sequence $(\alpha)$. Set $H^i(-) = H^i_{S_+}(-)$.

For the convenience of the reader we first give a detailed proof of $\dim H^3_{S_+}(S)^\vee \leq 1$ when $d \geq 4$.
We have $E_2^{3,0} = H^3(S)$.
Note $E^{3,0}_4 = E^{3,0}_\infty$. As $E^{3,0}_\infty$ is a sub-quotient  of $H^3_{S_+}(G, R) = 0$ we get that $E^{3,0}_\infty = 0$.
So we have a surjective map
$$ E_3^{0,2} \rt E_3^{3,0}  \rt 0.$$
We note that $ E_3^{0,2}$ is a sub-quotient of $ E_2^{0,2} = H^0(H^2(G, R))$ and so has finite length. Thus $E_3^{3,0}$ has finite length.
We have an exact    sequence
$$E_2^{1, 1} = H^1(H^1(G,R)) \rt E_2^{3,0} = H^3(S)  \rt E_3^{3,0} \rt 0. $$
Dualizing we get an exact sequence
$$ 0 \rt (E_3^{3,0})^\vee \rt H^3(S)^\vee \rt H^1(H^1(G, R))^\vee.$$
We have ${E_3^{3,0}}^\vee$ has finite length. By \ref{dual-dim} we get  $ \dim H^1(H^1(G, R))^\vee \leq 1$. The result follows.

Now let $4 \leq j \leq d -1$.  We have $E^{j,0}_2 = H^j_{S_+}(S)$.
Note $E^{j,0}_{j+1} = E^{j,0}_\infty$. As $E^{j,0}_\infty$ is a sub-quotient  of $H^j_{S_+}(G, R) = 0$ we get that $E^{j,0}_\infty = 0$.

(*) We also note that for $q > 0$ we have $E^{p,q}_r$ is a sub-quotient of $H^p(H^q(G,R))$. So $\dim (E^{p,q}_r)^\vee \leq p$ by \ref{dual-dim}.

We prove by descending induction that $\dim (E^{j,0}_t)^\vee \leq j -t$ for $2 \leq t \leq j$.
When $t =j$ we have an exact sequence $E^{0,j-1}_j \rt E_j^{j,0} \rt 0$. So  $\dim (E^{j,0}_t)^\vee \leq 0$. Now assume that the result holds for $t \geq 3$.
We have an exact sequence
$$E^{j-t+1, t-2}_{t-1} \rt E^{j,0}_{t-1} \rt E^{j,0}_t \rt 0.$$
Taking Matlis duals we have an exact sequence
$$ 0 \rt (E^{j,0}_t)^\vee \rt (E^{j,0}_{t-1})^\vee \rt (E^{j-t+1, t-2}_{t-1})^\vee.$$
By induction hypotheses we have $\dim (E^{j,0}_t)^\vee \leq j-t$. By (*) we get that \\
 $\dim (E^{j-t+1, t-2}_{t-1})^\vee \leq j-t + 1$. Thus
 $\dim (E^{j,0}_{t-1})^\vee \leq  j-t + 1$. This finishes the inductive step.

 So we have $\dim H^j(S)^\vee = \dim (E^{j,0}_2)^\vee \leq j - 2$. The result follows.
\end{proof}
Finally we give
\begin{proof}[Proof of Corollary \ref{khy}]
Assume $S_P$ is not \CM.
If $P$ is not a homogeneous prime ideal of $S$ then let $P^*$ be the prime ideal generated by all homogeneous elements of $P$. Then $\height P = \height P^* + 1$ ; see \cite[1.5.8]{BH}; and $\depth S_P = \depth S_{P^*} + 1$; see \cite[1.5.9]{BH}. Thus $S_{P^*}$ is not \CM. So it suffices to assume $P$ is a homogeneous prime ideal of $S$.

Let $P$ be a homogeneous prime ideal in $S$ with $S_P$ not \CM. Suppose if possible $(\db_{d,0}, \db_{d, 1}, \db_{d,2}) \nsubseteq P$. Let $x \in \{\db_{d,0}, \db_{d, 1}, \db_{d,2}\}$ be such that $x \notin P$. Let $D = D^*(d)$ be the Dickson algebra and let $D(u)$ be the $^*$-canonical module of $D$. Let $\q = P \cap D$. Then $x\notin \q$. We have $H^i_{S_+}(S)^\vee = \Ext^{d-i}_D(S, D(u))$ by local duality. As $x \notin \q$ and $x \in \sqrt{\ann \Ext^{i}_D(S, D(u))}$ for $i > 0$ it follows $\Ext^{i}_D(S, D(u))_\q = 0$ for $i > 0$. Thus  $S_\q$ is \CM. As $S_P$ is a localization of $S_\q$ it follows that $S_P$ is \CM. This is a contradiction. The result follows.

\end{proof}
\section{Appendix}
In the appendix we prove the following result.
\begin{theorem}
\label{dual-dim} Let $K$ be a field. Let $T= \bigoplus_{n \geq 0}T_n$ be a finitely generated graded $K = T_0$-algebra. Let $M$ be a finitely generated graded $T$-module of dimension $r$. Let $U_i = H^i_{S_+}(M)$  for $i \geq 0$. Note $U_i^{\vee}$ is a finitely generated graded $T$-module. Then $\dim U_i^\vee \leq i$ for $0 \leq i \leq r$.
\end{theorem}
\begin{remark}
This result is certainly known. However we have been unable to get a reference. As it is crucial we give a proof. When $T$ is standard graded the result follows from \cite[17.1.9]{bs}.
\end{remark}
\begin{proof}[Proof of Theorem \ref{dual-dim}]
Set $H^i(-)  = H^i_{T_+}(-)$.
We prove by induction on $i$ that if $M$ is a finitely generated graded $T$-module of dimension $s$ then $H^i(M)^\vee $ has dimension $\leq i$ for $i \geq 0$.
This result is certainly true when $i = 0$.
Assume $i \geq 1$ and that $\dim H^j(N) \leq j$ for all finitely generated graded $T$-modules with $j < i$.

If $E$ is a finitely generated, graded $T$-module then set $\ov{E} = E/H^0(E)$. If $\dim E > 0$ then $\depth \ov{E} > 0$ and $H^r(\ov{E}) = H^r(E)$ for $r > 0$. Let $x \in T_+$ be homogeneous such that $x$ is $\ov{E}$-regular. Then $\ker (E(-|x|) \xrightarrow{x} E)$ has finite length.

If $\dim M = 0$ then $H^r(M) = 0$ for $r > 0$. So we have nothing to show. Assume $\dim M > 0$.
Also if $\dim H^i(M)^\vee \leq 0$ then again  we have nothing to show. So assume $\dim H^i(M)^\vee > 0$. Let $x \in T_+$ be homogeneous which is $\ov{M} \oplus \ov{H^i(M)^\vee}$-regular.
We note that $H^j(M) \cong H^j(\ov{M})$ for all $j \geq 1$.
The exact sequence $0 \rt \ov{M}(-|x|) \xrightarrow{x} \ov{M} \rt N \rt 0$ induces a long
exact sequence
$$H^{i-1}(N) \xrightarrow{\delta} H^i(M)(-|x|) \xrightarrow{x} H^i(M).$$
Set $W = \image \delta$. Taking duals we get an exact sequence
$$ 0 \rt L \rt H^i(M)^\vee \xrightarrow{x} H^i(M)^\vee(|x|) \rt W^\vee \rt 0$$
and that $W^\vee $ is a submodule of $H^{i-1}(N)^\vee$ and so by induction hypothesis  has dimension $\leq i -1$. As $L$ has finite length it follows that $\dim H^i(M)^\vee \leq i$. The result follows by induction.
\end{proof}

We will also need the following property
\begin{theorem}\label{hom-sr}
 Let $K$ be a field. Let $T= \bigoplus_{n \geq 0}T_n$ be a finitely generated graded $K = T_0$-algebra. The following assertions are equivalent
\begin{enumerate}[\rm (i)]
\item $T$ satisfies $S_r$ property.
\item For all homogeneous primes $P$ of $T$ we have
$$ \depth T_P \geq \min\{\height P, r \}.$$
\end{enumerate}
\end{theorem}
\begin{proof}
The assertion (i) $\implies$ (ii) is trivial.  We prove the converse. Let $P$ be a prime ideal of $T$ which is not homogeneous.  Let $P^*$ be the prime ideal generated  by  all homogeneous elements of $P$. Then $\height P = \height P^* + 1$; see \cite[1.5.8]{BH} and $\depth T_P = \depth T_{P^*} + 1$; see \cite[1.5.9]{BH}. The result easily follows.
\end{proof}

\medskip

\section{Conflict of interest statement}
The authors have no conflict of interest to declare that are relevant to this article.

\section{Data availability statement}

\emph{The manuscript has no associated data.}

\end{document}